\documentclass{article}
\usepackage{amsmath,amssymb,amsthm}
\usepackage{amsfonts,amscd,latexsym,mathrsfs,verbatim}
\usepackage[usenames]{color}

\usepackage[colorlinks=true,linkcolor=blue,citecolor=blue,pdfpagelabels=false]{hyperref}

\theoremstyle{plain}
  \newtheorem{theorem}{Theorem}[section]
  \newtheorem{lemma}[theorem]{Lemma}
  \newtheorem{conjecture}[theorem]{Conjecture}
  
  \newtheorem{proposition}[theorem]{Proposition}
\theoremstyle{definition}
  \newtheorem{example}[theorem]{Example}
  \newtheorem{remark}[theorem]{Remark}
  \newtheorem{question}[theorem]{Question}
  
\newtheorem*{acknowledgements}{Acknowledgements}

\renewcommand{\geq}{\geqslant}
\renewcommand{\leq}{\leqslant}
\newcommand{\assign}{:=}
\newcommand{\pFq}[5]{\ensuremath{{}_{#1}F_{#2} \biggl( \genfrac{}{}{0pt}{}{#3}{#4} \biggm| {#5} \biggr)}}

\newcommand{\mathbbm}[1]{\mathbb{#1}}
\newcommand{\tmem}[1]{\emph{#1}}
\newcommand{\tmtextit}[1]{{\itshape{#1}}}
\newcommand{\tmtextup}[1]{{\upshape{#1}}}

\newcommand{\bx}{\boldsymbol{x}}
\newcommand{\bn}{\boldsymbol{n}}

\begin{document}

\title{Positivity of rational functions and their diagonals}
\author{Armin Straub\and Wadim Zudilin}

\maketitle

\hbox to\hsize{\hfill\sl To Dick Askey on the occassion of his birthday,\kern10mm}
\hbox to\hsize{\hfill\sl with many positive wishes\kern10mm}

\bigskip
\begin{abstract}
  The problem to decide whether a given rational function in several variables
  is positive, in the sense that all its Taylor coefficients are positive,
  goes back to Szeg\H{o} as well as Askey and Gasper, who inspired more recent
  work. It is well known that the diagonal coefficients of rational functions
  are $D$-finite. This note is motivated by the observation that, for several
  of the rational functions whose positivity has received special attention,
  the diagonal terms in fact have arithmetic significance and arise from
  differential equations that have modular parametrization. In each of these
  cases, this allows us to conclude that the diagonal is positive.

  Further inspired by a result of Gillis, Reznick and Zeilberger, we
  investigate the relation between positivity of a rational function and the
  positivity of its diagonal.
\end{abstract}

\noindent
{\small
{\em Keywords}: positivity; rational function; hypergeometric function; modular function;
Ap\'ery-like sequence; multivariate asymptotics.
}

\section{Introduction}

The question to decide whether a given rational function is
\tmtextit{positive}, that is, whether its Taylor coefficients are all
positive, goes back to Szeg\H{o} {\cite{szego-pos33}} and has since been
investigated by many authors including Askey and Gasper
{\cite{askey-pos74,askey-pos72,askey-pos77}}, Koornwinder
{\cite{koornwinder-pos78}}, Ismail and Tamhankar {\cite{ismail-pos79}},
Gillis, Reznick and Zeilberger {\cite{zb-pos-el83}}, Kauers
{\cite{kauers-pos07}}, Straub {\cite{straub-pos08}}, Kauers and Zeilberger
{\cite{kz-pos08}}, Scott and Sokal {\cite{ss-pos13}}. The interested reader
will find a nice historical account in {\cite{ss-pos13}}. A particularly
interesting instance is the Askey--Gasper rational function
\begin{equation}
  A (x, y, z) \assign \frac{1}{1 - x - y - z + 4 xyz}, \label{eq:AG3}
\end{equation}
whose positivity is proved in {\cite{askey-pos77}} and {\cite{zb-pos-el83}}.
Generalizations to more than three variables are rarely tractable, with the
longstanding conjecture of the positivity of
\begin{equation}
  \label{eq:E24} \frac{1}{1 - x - y - z - w + \frac{2}{3}  (xy + xz + xw + yz
  + yw + zw)},
\end{equation}
also referred to as the Lewy--Askey problem. Very recently, Scott and Sokal
\cite{ss-pos13} succeeded in proving the non-negativity of \eqref{eq:E24},
both in an elementary way by an explicit Laplace-transform formula and based on more general results on the basis
generating polynomials of certain classes of matroids.
Note that by a result from {\cite{kz-pos08}} the positivity of \eqref{eq:E24}
would follow from the positivity of
\begin{align}
  D (x, y, z, w) \assign \frac{1}{1 - x - y - z - w + 2 (yzw + xzw + xyw +
  xyz) + 4 xyzw} ,
\end{align}
which is still an open problem.
In another direction, Gillis, Reznick and Zeilberger conjecture in {\cite{zb-pos-el83}} that
\begin{equation}
  \frac{1}{1 - (x_1 + x_2 + \ldots + x_d) + d! x_1 x_2 \cdots x_d}
  \label{eq:rat-grz}
\end{equation}
has non-negative coefficients for any $d \geq 4$ (this is false for $d=2,3$). It is further asserted
(though the proof is ``omitted due to its length'') that, in order to show
the non-negativity of the rational functions in \eqref{eq:rat-grz}, it suffices to
prove that their {\tmem{diagonal}} Taylor coefficients are non-negative.
Modulo this claim, the cases $d = 4, 5, 6$ were established by Kauers
{\cite{kauers-pos07}}, who found and examined recurrences for the respective
diagonal coefficients.

The above claim from {\cite{zb-pos-el83}} suggests the following question.
Here, we denote by $e_k (x_1, \ldots, x_d)$ the elementary symmetric
polynomials defined by
\begin{equation}\label{eq:e}
  \prod_{j = 1}^d (x + x_j) = \sum_{k = 0}^d e_k (x_1, \ldots, x_d) x^{d - k} .
\end{equation}
\begin{question}
  \label{q:diag}Under what \tmtextup{(}natural\tmtextup{)}
  condition\tmtextup{(}s\tmtextup{)} is the positivity of a rational function
  $h (x_1, \ldots, x_d)$ of the form
  \begin{equation}
    h (x_1, \ldots, x_d) = \frac{1}{\sum_{k = 0}^d c_k e_k (x_1, \ldots, x_d)}
    \label{eq:rat-linear}
  \end{equation}
  implied by the positivity of its diagonal\tmtextup{?}
  For example, would the positivity of $h (x_1, \ldots, x_{d - 1}, 0)$ be a
  sufficient condition\tmtextup{?}
\end{question}

Another motivation for this question is the fact that for several important
rational functions, like the ones reproduced above, the diagonal coefficients
are arithmetically interesting sequences. In particular, expressing
them in terms of known hypergeometric summations sometimes makes their
positivity apparent. For instance, the diagonal sequence for $A (x, y, z)$
is
\begin{equation}
  a_{n, n, n} = \sum_{k = 0}^n \binom{n}{k}^3, \label{eq:AG3d}
\end{equation}
see Example~\ref{ex:AG-rat}, while the diagonal of $D (x, y, z, w)$ is given by
\begin{align}
d_{n, n, n, n} = \sum_{k = 0}^n \binom{n}{k}^2 \binom{2 k}{n}^2 ,
\end{align}
as is shown in Example~\ref{ex:KZ-rat4d}.
Since these diagonal sequences are manifestly positive, sufficient progress on
Question~\ref{q:diag} might provide a proof of the conjectured positivity of $D
(x, y, z, w)$.
More generally, proving positivity of a single sequence is much simpler from a
practical point of view than proving positivity of all Taylor coefficients of a
rational function, and tools such as cylindrical algebraic decomposition can be
used for this task rather successfully in specific examples, as illustrated by
\cite{kauers-pos07} and observed in some of the examples herein.

We note that, with no loss of generality, we may assume in
Question~\ref{q:diag} that $c_0 = 1$.

In Section \ref{sec:2d}, we answer Question~\ref{q:diag} in the affirmative
when $d = 2$. The three and four-dimensional cases are discussed in Sections
\ref{sec:3d} and \ref{sec:4d}, while Section~\ref{sec:5d} covers an approach
to positivity via asymptotics. In particular, we prove that the conjectural
conditions, given in \cite{straub-pos08}, for positivity of rational functions
in three variables are indeed necessary.

\section{The two-dimensional case}\label{sec:2d}

In the case of only two variables, we are thus interested in the rational
function
\[ h (x, y) = \frac{1}{1 + c_1 (x + y) + c_2 x y} . \]
Note that the condition $c_1<0$ is necessary to ensure the positivity of $h(x,y)$,
as $-c_1$ is a Taylor coefficient of its series expansion. The next examples
demonstrate that the positivity of the diagonal coefficients of $h(x,y)$ and
the positivity of $h(x,0)$ are not implied by each other.

\begin{example}
  The rational function $h(x,y) = 1 / (1 + x + y)$ has positive diagonal coefficients
  but is not positive; indeed, $h(x,0)=1-x+O(x^2)$. This illustrates that some
  condition is indeed needed in Question~\ref{q:diag}.
\end{example}

\begin{example}
  Let $h(x,y) = 1 / (1 - x - y + 2xy)$. Then $h(x,0)=1/(1-x)$ is positive,
  but the diagonal coefficients of $h(x,y)$ are not positive.
\end{example}

\begin{theorem}
\label{th:2d}
A two-variable rational function
\[ h (x, y) = \frac{1}{1 + c_1 (x + y) + c_2 x y} \]
is positive, if both $h(x,0)$ and the diagonal of $h(x,y)$ are positive.
\end{theorem}

\begin{proof}
The positivity of $h(x,0)$ implies that $c_1<0$. Upon rescaling the variables
by a positive factor, we may assume that $c_1=-1$ and write our rational
function in the form
\begin{equation}
  \frac{1}{1 - (x + y) + a x y} = \sum_{n,m=0}^\infty u_{n,m}x^ny^m. \label{eq:rat-2d}
\end{equation}
As demonstrated in the course of {\cite[Proposition 4]{straub-pos08}}, this
rational function is positive if and only if $a \leq 1$.
Here, we only need the easy observation that it is positive if $a\leq1$,
which follows directly from the geometric series and the factorization
$1-(x+y)+xy = (1-x)(1-y)$.

On the other hand, the diagonal terms $u_n \assign u_{n, n}$ of the Taylor
expansion \eqref{eq:rat-2d} are given by
\[ u_n = \sum_{k = 0}^n \frac{(2 n - k) !}{k! (n - k) !^2}
   (- a)^k . \]
We observe that the sequence $u_n$ is characterized by the generating series
\[ \sum_{n = 0}^{\infty} u_n z^n = \frac{1}{\sqrt{1 - 2 (2 - a) z + a^2 z^2}}
   . \]
For $a > 1$, the quadratic polynomial $1 - 2 (2 - a) z + a^2 z^2$ has non-real
roots, from which we conclude that $u_n$ is (eventually) sign-indefinite.
Therefore, the series \eqref{eq:rat-2d} is positive if and only if its
diagonal terms are positive.
\end{proof}

Theorem~\ref{th:2d} answers Question~\ref{q:diag} in the affirmative when $d = 2$.

\begin{remark}
  \label{rk:char-2d}The sequence $u_n$ satisfies the three-term recurrence
  \[ (n + 1) u_{n + 1} = (2 - a) (2 n + 1) u_n - a^2 n u_{n - 1}, \]
  which has characteristic polynomial $x^2 -2(2-a)x +a^2 = (x + a)^2 - 4 x$. Note that, for $a >
  1$, this polynomial has complex roots.
\end{remark}

The ultimate reduction to $d=1$ and $d=2$ performed in this section shows that
for $d\ge3$ we can normalize, without loss of generality, our $d$-variable rational function
\eqref{eq:rat-linear} to satisfy $c_0=1$, $c_1=-1$ and also $c_2=a\le1$.
This will be the canonical form of a rational function in Question~\ref{q:diag}.

\section{The three-dimensional case}\label{sec:3d}

A partially conjectural classification of positive rational functions of the
form
\begin{equation}
  \label{eq:hab} h_{a, b} (x, y, z) = \frac{1}{1 - (x + y + z) + a (xy + yz +
  zx) + bxyz}
\end{equation}
has been given in {\cite{straub-pos08}}. It is conjectured there
{\cite[Conjecture 1]{straub-pos08}} that $h_{a, b}$ is positive if and only if
the three inequalities $a \leq 1$, $b < 6 (1 - a)$,
$b \leq 2 - 3 a + 2 (1 - a)^{3 / 2}$ hold. In Theorem \ref{th:part1} below we show
that all three conditions are indeed necessary for positivity.

\begin{example}\label{ex:str-pos3d}
  It is proven in {\cite{straub-pos08}} that the rational
  function $h_{a, b}$ with
  \[ a = \frac{\lambda (\lambda + 2)}{(\lambda + 1)^2}, \hspace{1em} b = -
     \frac{(\lambda - 1)  (\lambda + 2)^2}{(\lambda + 1)^3} \]
  is positive for all $\lambda \geq 0$. The conjecture mentioned above
  predicts that it is, in fact, positive as long as $\lambda > (1 +
  \sqrt{2})^{1 / 3} - (1 + \sqrt{2})^{- 1 / 3} - 1 \approx - 0.403928$. Note
  that this rational function, after a scaling of variables, is
  \begin{align}
    \frac{1}{1 - (\lambda + 1) (x + y + z) + \lambda (\lambda + 2) (xy + yz +
    zx) - (\lambda - 1) (\lambda + 2)^2 xyz} .
  \end{align}
  Empirically, it appears that the Taylor coefficients of this rational function are
  polynomials in $\lambda$ with positive coefficients. We have verified this for the
  coefficients of $x^k y^m z^n$ with $k,m,n \leq 100$.
\end{example}

Note that, by the results of Section \ref{sec:2d}, $h_{a, b} (x, y, 0)$ is
positive if and only if $a \leq 1$. The next conjecture is therefore
equivalent to an affirmative answer to Question~\ref{q:diag} for $d = 3$.

\begin{conjecture}
  \label{conj:3d}Suppose that $a \leq 1$. Then the rational function~\eqref{eq:hab}
  is positive if and only if its diagonal is positive.
\end{conjecture}

In the case $a > 1$, which is not covered by Conjecture \ref{conj:3d},
the following is a conjectural characterization of the rational
functions which have positive diagonal coefficients.

\begin{conjecture}
  Let $a \geq 1$. The diagonal of~\eqref{eq:hab} is positive if and only if
  $b \leq - a^3$.
\end{conjecture}

That the case $b = - a^3$ plays a special role can be seen from the
characteristic polynomial of the recurrence of minimal order for the diagonal
coefficients. Namely, the diagonal coefficients $u_n$ of $h_{a, b}$, as
defined in \eqref{eq:hab}, satisfy a fourth-order recurrence (the coefficients
have degree $4$ in $n$, degree $9$ in $a$ and degree $5$ in $b$), whose
characteristic polynomial is
\begin{equation}
  (a^3 + b) (a^3 + a b - (1 - a) x) ((x + b)^3 + 27 x (a^3 + a b - (1 - a) x))
  . \label{eq:3d-char4}
\end{equation}
Obviously, the first factor vanishes when $b = - a^3$. We further observe that
the cubic factor in \eqref{eq:3d-char4} has discriminant
\[ - 3^9 (a^3 - 3 a^2 - b)^2 (4 a^3 - 3 a^2 + 6 a b + b^2 - 4 b) . \]
The second factor vanishes if and only if $b = 2 - 3 a \pm 2 (1 - a)^{3 / 2}$,
which includes the principal condition of {\cite[Conjecture 1]{straub-pos08}}
as stated at the beginning of this section; see also Example \ref{eg:ns-3d}.


In the remainder of this section, we consider two cases of particular interest
in the three-variable case, namely the Askey--Gasper rational function $A (x,
y, z)$ from \eqref{eq:AG3} as well as Szeg\H{o}'s rational function
\begin{equation}\label{eq:S}
  S (x, y, z) = \frac{1}{1 - (x + y + z) + \frac{3}{4}  (xy + yz + zx)} .
\end{equation}
In both cases, we exhibit the arithmetic nature of the diagonals and
demonstrate that positivity follows as a consequence.

\begin{example}\label{ex:AG-rat}
  According to {\cite{ez-diag}}, the diagonal sequence $a_n \assign a_{n, n, n}$ for
  $A (x, y, z)$ as in \eqref{eq:AG3} satisfies the three-term recurrence equation
  \[ (n + 1)^2 a_{n + 1} = (7 n^2 + 7 n + 2) a_n + 8 n^2 a_{n - 1} , \]
  the latter already implying the positivity of the diagonal.
  The same recursion holds for the sequence $a_n = \sum_{k = 0}^n
  \binom{n}{k}^3$, also known as Franel numbers, which is a classical result
  of Franel {\cite{franel94}}. Furthermore, both sequences have the same initial conditions $a_0 = 1$,
  $a_1 = 2$. As indicated in \eqref{eq:AG3d}, the diagonal
  of $A (x, y, z)$ is thus given by the Franel numbers. The sequence $a_n$ is
  an Apery-like sequence, namely {\cite[sequence
  (4.8){\hspace{0.25em}}(a)]{asz-clausen}}, and the (Calabi--Yau) differential
  equation satisfied by the generating function has modular parametrization.
  As a consequence, the generating function has a hypergeometric form, namely
  \begin{equation}
    \sum_{n = 0}^{\infty} a_n z^n = \frac{1}{1 - 2 z}
    \,\pFq21{\frac13,\frac23}{1}{\frac{27z^2}{(1-2z)^3}},
  \label{fran}
  \end{equation}
  which we record for comparison with the next example. Here and in what follows
  \begin{equation*}
    \,\pFq21{a,b}{c}{z} = \sum_{n = 0}^{\infty} z^n \prod_{j=0}^{n-1}\frac{(a+j)(b+j)}{(1+j)(c+j)}
  \end{equation*}
  is the hypergeometric function.
  Note that positivity of $a_n$ is still apparent from~\eqref{fran}.
\end{example}

\begin{example}\label{ex:Sze-rat}
  As shown in {\cite{straub-pos08}}, the positivity of $S (x, y, z)$ can be deduced
  from the positivity of $A (x, y, z)$. On the other hand, the diagonals of the two
  are not related to each other in an easy way;
  the diagonal terms $s_n = [(xyz)^n] S (2 x, 2 y, 2 z)$ are given by
  \[ 1, 12, 198, 3720, 75690, 1626912, \ldots \]
  and satisfy the recurrence
  \begin{equation}
    2 (n + 1)^2 s_{n + 1} = 3 \left( 27 n^2 + 27 n + 8 \right) s_n - 81 (3 n -
    1)  (3 n + 1) s_{n - 1} . \label{eq:Sd-rec}
  \end{equation}
  Denoting by $y_0 (z) = \sum_{n \geq 0} s_n z^n$ the generating
  function of this sequence, it is routine to verify that
  \begin{equation}
    y_0 (z) = \pFq21{\frac13,\frac23}{1}{27z(2-27z)};
  \label{eq:Sd-gf}
  \end{equation}
  indeed, both sides in \eqref{eq:Sd-gf} satisfy the same differential equation and initial values.
  See Remark \ref{rk:Sd-gf} below on how one can find this expression. We note that
  \eqref{eq:Sd-gf} implies the binomial formula
  \begin{align}
    s_n = \sum_{k = 0}^n (- 27)^{n - k} 2^{2 k - n}  \frac{(3 k) !}{k!^3}
    \binom{k}{n - k},
  \end{align}
  though positivity is not apparent here.
\end{example}

\begin{lemma}\label{lem:Sze-rat}
The sequence $s_n$ in Example~\textup{\ref{ex:Sze-rat}} is positive.
\end{lemma}

\begin{proof}
  To deduce positivity of $y_0 (z)$ in \eqref{eq:Sd-gf}, start with
  Ramanujan's cubic transformation \cite[p.~97]{berndtV}
  \begin{equation*}
    \pFq21{\frac13,\frac23}{1}{1 - \left( \frac{1 - x}{1 + 2 x} \right)^3}
    = (1 + 2 x) \pFq21{\frac13,\frac23}{1}{x^3},
  \end{equation*}
  which is proven in {\cite{bb-cubicagm}}; see also \cite[Corollary 6.2]{maier-hyp07}.
  With $x$ and $z$ related by
  \[ 27 z (2 - 27 z) = 1 - \left( \frac{1 - x}{1 + 2 x} \right)^3, \]
  we find that
  \[ 2 x (z) = \frac{3}{1 + 2 (1 - 27 z)^{2 / 3}} - 1. \]
  The binomial theorem shows that $(1 - 27 z)^{2 / 3} = 1 - z g (z)$ for some
  $g (z)$ with positive Taylor coefficients. It follows that $x (z) = c_1 z +
  c_2 z^2 + \ldots$ for positive $c_j$, so that
  \[ y_0 (z) = (1 + 2 x (z)) \pFq21{\frac13,\frac23}{1}{x(z)^3} \]
  is seen to have positive coefficients.
\end{proof}

\begin{remark}
  \label{rk:Sd-gf}Let us briefly indicate how we found \eqref{eq:Sd-gf}.
  First, note that $y_0 (z)$ is the analytical solution of the differential equation
  corresponding to \eqref{eq:Sd-rec} characterized by $y_0 (0) = 1$.
  Let $y_1 (z)$ be the solution such that $y_1 (z) - y_0
  (z) \log (z) \in z\mathbbm{Q} [[z]]$. Then
  \[ q (z) \assign \exp \left( \frac{y_1 (z)}{y_0 (z)} \right) = z + \frac{33
     z^2}{2} + 306 z^3 + \frac{12203 z^4}{2} + 128109 z^5 + O \left( z^6
     \right) . \]
  Denoting by $z (q)$ the inverse function, we observed, by computing the
  first few terms of the $q$-expansion, that
  \[ y_0  (z (q / 2)) = \sum_{n, m \in \mathbbm{Z}} q^{n^2 + nm + m^2} . \]
  The right-hand side is the theta series of the planar hexagonal lattice,
  also known as the first cubic theta function $a (q)$, and its relation to
  the hypergeometric function in \eqref{eq:Sd-gf} is well known; see, for
  instance, {\cite{bb-cubicagm}}. For further background on this approach we
  refer the interested reader to {\cite{asz-clausen}}.
\end{remark}

\section{Examples in the four-dimensional case}\label{sec:4d}

We now study positivity of the rational functions in four variables, which are
of type \eqref{eq:rat-linear}. That is, we consider the rational functions
\begin{equation}
  \label{eq:habc} h_{a, b, c} (\bx) = \frac{1}{1 - e_1(\bx) + ae_2(\bx) + be_3(\bx) + ce_4(\bx)},
\end{equation}
where $\bx = (x_1,x_2,x_3,x_4)$ and $e_k(\bx)$ are the elementary symmetric
functions defined in \eqref{eq:e}.
Table~\ref{tbl:4d} summarizes the examples we discuss in this section.

\begin{table}[h]
  \begin{center}
  \begin{tabular}{|c|c|c|l|}
    \hline
    $a$ & $\phantom+b$ & $\phantom+c$ & \\
    \hline
    $0$ & $\phantom+2$ & $\phantom+4$ & positivity conjectured in {\cite{kz-pos08}}\\
    \hline
    $\frac{2}{3} \vphantom{\big|^0}$ & $\phantom+0$ & $\phantom+0$ & positivity conjectured in {\cite{askey-pos72}}; implied by $h_{0, 2, 4}$;\\[-.5mm]
    & & & \quad non-negativity proven in {\cite{ss-pos13}}\\
    \hline
    $0$ & $\phantom+\frac{64}{27} \vphantom{\big|^0_0}$ & $\phantom+0$ & positivity conjectured in {\cite{kauers-pos07}}\\
    \hline
    $0$ & $\phantom+0$ & $\phantom+24$ & non-negativity conjectured in {\cite{zb-pos-el83}};\\[-.5mm]
    & & & \quad partly proven in {\cite{kauers-pos07}}\\
    \hline
    $0$ & $\phantom+4$ & $-16$ & positivity proven in {\cite{koornwinder-pos78}}; see also
    {\cite{zb-pos-el83}}\\
    \hline
    $\frac{8}{9} \vphantom{\big|^0_0}$ & $-\frac{16}{27}$ & $\phantom+0$ & positivity proven in
    {\cite[{\S}3]{szego-pos33}}; implied by $h_{0, 4, - 16}$\\
    \hline
  \end{tabular}
  \caption{\label{tbl:4d}Interesting instances of $h_{a, b, c}$ as in
  \eqref{eq:habc}}
  \end{center}
\end{table}

\begin{example}[Lewy--Askey rational function]\label{ex:AG-rat4d}
  In {\cite{askey-pos72}}, Askey and Gasper mention the following
  four-dimensional generalization of Szeg\H{o}'s function
  \[ h_{2 / 3, 0, 0}(x,y,z,w) = \frac{1}{1 - (x + y + z + w) + \frac{2}{3} (xy + xz + xw
     + yz + yw + zw)}, \]
  which is just a rescaled version of $1 / e_2(1 - x, 1 - y, 1 - z, 1 - w)$.
  As already mentioned in the introduction, the non-negativity of this rational function
  was recently established in {\cite{ss-pos13}}. The scaled initial
  diagonal terms $s_n := 9^n [(x y z w)^n] h_{2 / 3, 0, 0} (x, y, z, w)$ are
  \[ 1, 24, 1080, 58560, 3490200, 220739904, \ldots, \]
  and one checks that $s_n = \binom{2n}n u_n$, where the sequence $u_n$ satisfies the recurrence equation
  \[ 3 (n + 1)^2 u_{n + 1} = 4 \left( 28 n^2 + 28 n + 9 \right) u_n - 64 (4 n -
     1) (4 n + 1) u_{n - 1} . \]
  As in Example \ref{ex:Sze-rat}, the differential equation, of which the
  generating function $y_0 (z) = \sum_{n \geqslant 0} u_n z^n$ is the unique
  analytical solution with value $1$ at $z = 0$, admits modular parametrization.
  This fact was found and communicated to us by van Straten. As a consequence,
  we have the hypergeometric representation
  \begin{equation}
    y_0 (z) = \frac{1}{(1 - 48 z + 12288 z^3)^{1 / 4}}
    \,\pFq21{\frac1{12},\frac5{12}}{1}{\frac{-1728 z^2 (3 - 64 z) (1 - 16 z)^6}{(1 - 48 z + 12288 z^3)^3}},
  \label{duco}
  \end{equation}
  which, once found, can be verified by comparing the differential equations
  satisfied by both sides. In fact, using hypergeometric transformations, we
  find that \eqref{duco} simplifies to
  \begin{equation}\label{ducox}
    y_0(z) = \frac{1}{\sqrt{1-24z}} \,\pFq21{\frac14,\frac34}{1}{\frac{-64z^2(3-64z)}{(1-24z)^2}} .
  \end{equation}
  As in Example \ref{ex:Sze-rat}, we can now use the arithmetic properties of
  this function to show that the sequence $u_n$, hence the diagonal terms
  $s_n$, are indeed positive.
  To do so, we may proceed as in Lemma \ref{lem:Sze-rat}, except now using
  Ramanujan's quadratic transformation {\cite[p. 146]{berndtV}} (also proven in \cite[Corollary 6.2]{maier-hyp07}).
  Alternatively, we can prove positivity of the diagonal terms from the above
  three-term recurrence using cylindrical algebraic decomposition in the style
  of \cite{kauers-pos07}.
\end{example}


\begin{example}[Kauers--Zeilberger rational function]
\label{ex:KZ-rat4d}
  On the other hand, the positivity of the rational function in the previous example is
  implied, as shown in {\cite{kz-pos08}} using positivity preserving operators,
  by the (conjectured) positivity of the rational function
  \begin{equation*}
    h_{0, 2, 4}(x,y,z,w) = \frac{1}{1 - (x + y + z + w) + 2 e_3(x,y,z,w) + 4xyzw},
  \end{equation*}
  which we also refer to as $D (x, y, z, w)$. This rational function, as
  mentioned in the introduction, has particularly appealing diagonal
  coefficients. Namely, expanding
  \begin{equation*}
    D(x,y,z,w) = \sum_{n=0}^\infty \left[ (x + y + z + w) - 2 (yzw + xzw + xyw + xyz) - 4
     xyz w \right]^n,
  \end{equation*}
  and applying the binomial theorem, one obtains a five-fold sum for the
  diagonal coefficients $d_n = d_{n,n,n,n}$.
  With the help of the multivariate Zeilberger algorithm {\cite{apa-zeil}} we verify that
  the sequence $d_n$ satisfies, for $n = 1, 2, \ldots$,
  \[ (n + 1)^3 d_{n + 1} - 4 (2 n + 1)  (3 n^2 + 3 n + 1) d_n + 16 n^3 d_{n -
     1} = 0. \]
  The same recurrence is satisfied by {\cite[sequence (4.12){\hspace{0.25em}}($\epsilon$)]{asz-clausen}},
  so that comparing initial values proves that
  \[ d_{n} = \sum_{k = 0}^n \binom{n}{k}^2 \binom{2 k}{n}^2 . \]
  Since this makes positivity of the diagonal terms obvious, an affirmative
  answer to Question~\ref{q:diag} when $d=4$ would prove the conjectured
  positivity of $D(x,y,z,w)$.
\end{example}

\begin{example}[Examples of Szeg\H{o} and Koornwinder]\label{ex:SK-rat4d}
  Szeg\H{o} proved, as a higher-order generalization of the function \eqref{eq:S}, the positivity
  of $1 / e_3(1 - x, 1 - y, 1 - z, 1 - w)$;
  see {\cite[{\S}\,3]{szego-pos33}}. Upon rescaling, this is $h_{8 / 9, - 16 / 27, 0}$.
  The positivity of the rational function $h_{8 / 9, - 16 / 27, 0}$ can also be
  obtained, again via positivity preserving operators, see
  {\cite{straub-pos08}}, from the positivity of Koornwinder's rational function~{\cite{koornwinder-pos78}}
  \begin{equation*}
    h_{0, 4, -16}(x,y,z,w) = \frac{1}{1 - (x + y + z + w) + 4 e_3(x,y,z,w) -16xyzw},
  \end{equation*}
which was proven in~{\cite{zb-pos-el83,koornwinder-pos78}}.
  Using the multivariate Zeilberger algorithm {\cite{apa-zeil}}, as in Example~\ref{ex:KZ-rat4d},
  we can show that the diagonal of $h_{0, 4, - 16}$ is given by the, obviously
  positive, sequence
  \[ \sum_{k = 0}^n \binom{2 k}{k}^2 \binom{2 (n - k)}{n - k}^2 . \]
  This, again, is an Ap\'ery-like sequence, namely {\cite[sequence
  (4.10){\hspace{0.25em}}($\beta$)]{asz-clausen}}.
\end{example}

\begin{example}\label{ex:GRZ-rat4d}
  As shown in {\cite{kauers-pos07}}, modulo the assertion
  from {\cite{zb-pos-el83}} with an unpublished proof, the rational function
  \[ h_{0, 0, c}(x,y,z,w) = \frac{1}{1 - (x + y + z + w) + c x y z w} \]
  has non-negative coefficients if and only if $c \leq 24$. The condition
  $c \leq 24$ is necessary because the (diagonal) coefficient of
  $x y z w$ in $h_{0, 0, c}$ is $24 - c$. On the other hand,
  asymptotic considerations, such as in Example~\ref{eg:ns-4d} below, suggest that all
  but finitely many diagonal coefficients of $h_{0, 0, c}$ are positive if $c < 27$.

  This is the case $d = 4$ of a general conjecture from {\cite{zb-pos-el83}} mentioned in
  the introduction. In the general case, the rational function
  \[ \frac{1}{1 - (x_1 + x_2 + \ldots + x_d) + c x_1 x_2 \cdots x_d} \]
  can only be non-negative if $c \leq d!$ since the coefficient of $x_1 x_2
  \cdots x_d$ is seen to be $d! - c$. It is conjectured in {\cite{zb-pos-el83}}
  that, for $d\geq4$, the condition $c \leq d!$ is indeed sufficient, and it is
  claimed that the non-negativity follows from the non-negativity of the diagonal.
  On the other hand, some hypergeometric intuition suggests that the diagonal
  coefficients are eventually positive if $c < (d - 1)^{d - 1}$.
\end{example}

\begin{example}\label{ex:Kau-rat4d}
  In \cite{kauers-pos07}, the rational function
  \begin{equation*}
    h_{0,64/27,0}(x,y,z,w)
    = \sum_{k,l,m,n \ge 0} u_{k,l,m,n} x^k y^l z^m w^n
  \end{equation*}
  is conjectured to be positive. As evidence, it is shown in \cite{kauers-pos07},
  using cylindrical algebraic decomposition (CAD), that $u_{k,l,m,n} > 0$
  whenever the sum of the smallest two indices is at most $12$.
  On the other hand, we find that the diagonal coefficients satisfy a recurrence
  of order $3$ and degree $6$.  While we were unable to discover any closed form
  expression for the diagonal terms, we have used CAD to prove that they are
  positive. Once more, an affirmative answer to Question~\ref{q:diag} when
  $d=4$ would therefore imply the conjectured positivity of $h_{0,64/27,0}$.
\end{example}

\section{Multivariate asymptotics}
\label{sec:5d}

Multivariate asymptotics, as developed in {\cite{bp-masy3,pw-masy1,pw-masy2,rw-masy}}
and further illustrated
in {\cite{pw-masy-eg}}, is an approach to determine the asymptotics of the
coefficients $u_{n_1, \ldots, n_d}$ of a multivariate generating function
\[ h (x_1, \ldots, x_d) = \sum_{n_1, \ldots, n_d \geq 0} u_{n_1, \ldots,
   n_d} x_1^{n_1} \cdots x_d^{n_d} \]
directly from $h$ and its singular points.

In the sequel, we write $\bx = (x_1, \ldots, x_d)$. In the cases we are
presently interested in, $h = 1 / p$ is the reciprocal of a polynomial $p (\bx)$.
Denote with $\mathcal{V} \subseteq \mathbbm{C}^d$ the {\tmem{singular
variety}} defined by $p = 0$. A point $\bx \in \mathcal{V}$ is smooth if $\nabla
p (\bx) = (\partial_1 p (\bx), \ldots, \partial_d p (\bx)) \neq \bold0$, where
$\partial_j \assign \partial/\partial x_j$ for $j=1,\dots,d$.
The nonsmooth points can be comfortably computed using Gr\"obner bases, as detailed in
{\cite[Section 4]{pw-masy-eg}}.

The next three examples suggest that rational functions which are on the
boundary of positivity (that is, slightly perturbing its coefficients can change
whether the function is positive) are intimately linked with rational
functions that have nonsmooth points on their singular variety. This echoes
the remark in {\cite{pw-masy-eg}} that, while for generic functions all points
of the singular variety are smooth, ``interesting applications tend not to be
generic.''

\begin{example}
  With $d = 2$, consider the case
  \[ h (x_1, x_2) = \frac{1}{1 - (x_1 + x_2) + a x_1 x_2}, \]
  which by Theorem~\ref{th:2d} is positive if and only if $a\leq1$.
  Then the singular variety has nonsmooth points if and only if $a = 1$. The
  nonsmooth point in the case $a = 1$ is $\bx = (1, 1)$.

  If $a < 1$ we may apply the machinery of {\cite{pw-masy1}} to find that
  \[ u_{n, n} \sim \frac{(1 + \sqrt{1 - a})^{2 n + 1}}{2 \sqrt{\pi n \sqrt{1 -
     a}}} . \]
\end{example}

\begin{example}
  \label{eg:ns-3d}With $d = 3$, consider the case
  \[ h (x_1, x_2, x_3) = \frac{1}{1 - (x_1 + x_2 + x_3) + a (x_1 x_2 + x_2 x_3
     + x_3 x_1) + b x_1 x_2 x_3} . \]
  Then the singular variety has nonsmooth points if and only if
  \[ 4 a^3 - 3 a^2 + 6 a b + b^2 - 4 b = 0. \]
  Solving this condition for $b$, gives
  $b = 2 - 3 a \pm 2 (1 - a)^{3 / 2}$,
  which includes precisely the boundary in {\cite[Conjecture 1]{straub-pos08}}
  explicitly describing the transition between positive rational functions and
  those with negative coefficients.
\end{example}

\begin{example}
  \label{eg:ns-4d}With $d = 4$, consider the case
  \[ h (x_1, x_2, x_3, x_4) = \frac{1}{1 - e_1 (\bx) + a e_2 (\bx) + b e_3 (\bx) + c
     e_4 (\bx)} . \]
  Then the singular variety has nonsmooth points if and only if
  \begin{align}
    0 & = (a^3 + 2 a b - a c + b^2 + c) (64 b^3 - 27 (b^4 + c^2) + 6 b c (2
    c - b) + c^3 \nonumber\\
    &\qquad - 54 a (2 b - c) (b^2 + c) + 18 a^2 (2 b^2 + 10 b c - c^2) - 54 a^3
    (b^2 + c) + 81 a^4 c) .
  \label{eq:factorization}
  \end{align}
  We note that all the examples in Table~\ref{tbl:4d}, with the exception of
  $(0, 0, 24)$ which is not `natural' as pointed out in Example~\ref{ex:GRZ-rat4d}, have nonsmooth points on the singular variety. The above
  factorization of the right-hand side of \eqref{eq:factorization} implies that, when
  \begin{equation}
    c = \frac{a^3 + 2 a b + b^2}{a - 1}, \label{eq:ns-4d-c}
  \end{equation}
  the rational function $h$ has nonsmooth points on its singular variety.
  The examples $(a,b,c)=(0, 4, - 16)$ and $(8 / 9, - 16 / 27, 0)$ from
  Table~\ref{tbl:4d} are of this form.
\end{example}

\begin{example}
  In the case $a = 0$ of (\ref{eq:ns-4d-c}), the rational function is
  \[ h_{0, b, - b^2} = \frac{1}{1 - (x + y + z + w) + b (yzw + xzw + xyw +
     xyz) - b^2 x y z w} . \]
  By directly computing the coefficients from their recurrence, we observe that its Taylor coefficients
  $u_{k,l,m,n}$ are positive for all $0 \leq k,l,m,n \leq 20$ if and only
  if $b < 4.00796 \ldots$, which suggests that $h_{0, b, - b^2}$ is positive
  if and only if $b \leq 4$. Note that positivity was proven in
  {\cite{koornwinder-pos78}} for the case $b = 4$. On the other hand, upon
  setting, say, $w = 0$, it follows from {\cite[Proposition 5]{straub-pos08}}
  that positivity of $h_{0, b, - b^2}$ requires $b \leq 4$. However,
  it appears empirically that the diagonal coefficients are positive for any $b \in
  \mathbbm{R}$. We have verified this for the first $50$ coefficients.
  Note that this fits nicely and further illustrates
  Question \ref{q:diag}. We note, however, that, for $b < 4$, the function
  $h_{0, b, - b^2}$ does not appear to be on the boundary of positivity.
\end{example}

Let a direction $\bn = (n_1, \ldots, n_d) \in \mathbbm{Z}_{> 0}^d$ be given.
Among points on the singular variety $\mathcal{V}$, a special role is played
by the {\tmem{critical points}} $\bx$ for $\bn$, which are characterized
{\cite[Proposition 3.11]{pw-masy-eg}} by the $d$ equations $p (\bx) = 0$ and,
for all $j = 1, \ldots, d - 1$,
\begin{equation}
  n_d x_j \partial_j p (\bx) = n_j x_d \partial_d p (\bx) . \label{eq:cp-def}
\end{equation}
We note the following consequence of {\cite[Theorem 3.16]{pw-masy-eg}}, which
is a simple reformulation of {\cite[Proposition 5.1]{rw-masy}}; see the remark
after {\cite[Theorem 3.16]{pw-masy-eg}} for the uniqueness.

\begin{proposition}
  \label{prop:cp-nn}Let $\bn \in \mathbbm{Z}_{> 0}^d$ be such that there is a
  smooth critical point for~$\bn$. If the rational function $h$ is non-negative,
  then there is a unique critical point for~$\bn$ in $\mathbbm{R}_{> 0}^d$.
\end{proposition}

\begin{example}
\label{ex:4b}
  For illustration, we use Proposition~\ref{prop:cp-nn} to give an alternative
  proof of {\cite[Proposition 5]{straub-pos08}}, which states that the
  rational function
  \[ h (x, y, z) = \frac{1}{1 - (x + y + z) + b x y z} \]
  is not non-negative if $b > 4$. Suppose that $b > 4$, in which case
  $\mathcal{V}$ is smooth. A simple computation shows that $h$ has three
  critical points for $\bn = (1, 1, 1)$, namely, the points $(c, c, c)$, where $c$
  is a solution of
  \begin{equation}
    1 - 3 c + b c^3 = 0. \label{eq:f3b-c}
  \end{equation}
  This cubic equation has discriminant $\Delta = 27 b (4 - b)$. The assumption
  $b > 4$ implies that $\Delta < 0$, which in turn implies that the equation
  \eqref{eq:f3b-c} has only one real root. Since this real root is necessarily
  negative by the intermediate value theorem, we conclude that none of the
  three critical points for $(1, 1, 1)$ lies in $\mathbbm{R}_{> 0}^3$. By
  Proposition~\ref{prop:cp-nn} it follows that $h$ is not non-negative if $b > 4$.
\end{example}

Generalizing the approach taken in this example, we are able to show one part
of {\cite[Conjecture 1]{straub-pos08}}.

\begin{theorem}\label{th:part1}
  For the rational function
  \[ h (x, y, z) = \frac{1}{1 - (x + y + z) + a (x y + y z + z x) + b x y z}
  \]
  to be non-negative it is necessary that $a \leq 1$ and $b \leq 2 -
  3 a + 2 (1 - a)^{3 / 2}$.
\end{theorem}

\begin{proof}
  Upon setting $z = 0$, it follows from Section \ref{sec:2d} that $a \leq
  1$ is a necessary condition for $h$ to be non-negative.

  The defining equations \eqref{eq:cp-def} for critical points for $\bn = (1, 1,
  1)$ are equivalent to
  \[ (x - z) (a y - 1) = 0, \hspace{1em} (y - z) (a x - 1) = 0. \]
  There are therefore two kinds of critical points. Firstly, the points $(c,
  c, c)$, where $c$ is a solution to
  \begin{equation}
    1 - 3 c + 3 a c^2 + b c^3 = 0; \label{eq:f3ab-c}
  \end{equation}
  secondly, the points $(x, y, z)$ where two coordinates are equal to $1 / a$
  and the third coordinate is $a (1 - a) / (a^2 + b)$.

  We observe that the discriminant of the cubic equation \eqref{eq:f3ab-c} is
  negative if $a \leq 1$ and $b > 2 - 3 a + 2 (1 - a)^{3 / 2}$. In that
  case, by the same argument as in the previous example, there are no critical
  points of the first kind in $\mathbbm{R}_{> 0}^3$. On the other hand, unless
  $b = - a^3$, there are three distinct critical points of the second kind,
  which either all lie in $\mathbbm{R}_{> 0}^3$ or all lie outside
  $\mathbbm{R}_{> 0}^3$.

  Suppose that $a \leq 1$ and $b > 2 - 3 a + 2 (1 - a)^{3 / 2}$. Then the
  case $b = - a^3$ occurs only if $a < - 3$, in which case the critical points
  of the second kind lie outside $\mathbbm{R}_{> 0}^3$. We conclude that $h$
  cannot have a unique critical point in $\mathbbm{R}_{> 0}^3$. Since, by
  Example~\ref{eg:ns-3d}, all points on the singular variety of $h$ are
  smooth, the claim therefore follows from Proposition~\ref{prop:cp-nn}.
\end{proof}

Note the special role played by the diagonal direction $\bn = (1,1,1)$, though
the present proof does use global information when applying
Proposition~\ref{prop:cp-nn}.

\medskip
Finally, we would like to point out a resemblance of Question~\ref{q:diag}
and the asymptotic analysis in this section with a famous theorem of Grace {\cite{grace,hormander}},
also known as the Grace--Walsh--Szeg\H{o} coincidence theorem. It states that,
given a simply connected open set $D\subset\mathbb C$,
if either $D$ is convex or the degree of the polynomial
$$
p (\bx) = p (x_1, \ldots, x_d) = \sum_{k = 0}^d c_k e_k (\bx)
$$
is $d$, then, for any $\xi_1,\ldots,\xi_n\in D$, there exists $\xi\in D$ such that
$$
p(\xi_1,\ldots,\xi_n) = p(\xi,\ldots,\xi).
$$

\begin{acknowledgements}
We are indebted to Duco van Straten for providing us with the explicit formula~\eqref{duco}
and also with geometric insights behind his derivation. We thank the referees for their healthy
criticism on an earlier version of this work.
\end{acknowledgements}


\bigskip
\noindent
Armin Straub:
{\sc Max-Planck-Institut f\"ur Mathematik, Vivatsgasse 7, Bonn D-53111, Germany} \\
{\em URL}: \texttt{http://arminstraub.com/} \\
{\em Current address}: {\sc Department of Mathematics, University of Illinois, 1409 West Green St, Urbana, IL 61801, USA}

\medskip
\noindent
Wadim Zudilin:
{\sc School of Mathematical and Physical Sciences,
The University of Newcastle, Callaghan NSW 2308, Australia} \\
{\em URL}: \texttt{http://wain.mi.ras.ru/}


\begin{thebibliography}{99}

\bibitem{asz-clausen}
\textsc{G.~Almkvist}, \textsc{D.~van Straten} and \textsc{W.~Zudilin},
Generalizations of {C}lausen's formula and algebraic transformations of {C}alabi--{Y}au differential equations,
\newblock {\em Proc. Edinburgh Math. Soc.} \textbf{54} (2011), 273--295.

\bibitem{apa-zeil}
\textsc{M.~Apagodu} and \textsc{D.~Zeilberger},
Multi-variable Zeilberger and Almkvist--Zeilberger algorithms and the sharpening of Wilf--Zeilberger theory,
\emph{Adv. in Appl. Math.} \textbf{37} (2006), 139--152.

\bibitem{askey-pos74}
\textsc{R.~Askey},
Certain rational functions whose power series have positive coefficients. {II},
{\em SIAM J. Math. Anal.} \textbf{5} (1974), 53--57.

\bibitem{askey-pos72}
\textsc{R.~Askey} and \textsc{G.~Gasper},
Certain rational functions whose power series have positive coefficients,
{\em Amer. Math. Monthly} \textbf{79} (1972), 327--341.

\bibitem{askey-pos77}
\textsc{R.~Askey} and \textsc{G.~Gasper},
Convolution structures for Laguerre polynomials,
{\em J. Analyse Math.} \textbf{31} (1977), 48--68.

\bibitem{bp-masy3}
\textsc{Y.~Baryshnikov} and \textsc{R.~Pemantle},
Asymptotics of multivariate sequences, part III: Quadratic points,
{\em Adv. Math.} \textbf{228} (2011), 3127--3206.

\bibitem{berndtV}
\textsc{B.\,C. Berndt},
{\em Ramanujan's Notebooks, Part~V}
(Springer-Verlag, New York, 1998).

\bibitem{bb-cubicagm}
\textsc{J.\,M. Borwein} and \textsc{P.\,B. Borwein},
A cubic counterpart of {J}acobi's identity and the {AGM},
{\em Trans. Amer. Math. Soc.} \textbf{323} (1991), 691--701.

\bibitem{ez-diag}
\textsc{S.\,B. Ekhad},
A linear recurrence equation for the diagonal coefficients of the power series of $1/(1-x-y-z+4xyz)$,
{\em Preprint} (2012), 1~p.;
available from
\texttt{http://\allowbreak www.math.rutgers.edu/\allowbreak\~{}zeilberg/\allowbreak tokhniot/\allowbreak oMultiAlmkvistZeilberger3}.

\bibitem{franel94}
\textsc{J.~Franel},
On a question of {L}aisant,
{\em L'Interm\'ediaire des Math\'ematiciens} \textbf{1} (1894), 45--47.

\bibitem{zb-pos-el83}
\textsc{J.~Gillis}, \textsc{B.~Reznick} and \textsc{D.~Zeilberger},
On elementary methods in positivity theory,
{\em SIAM J. Math. Anal.} \textbf{14} (1983), 396--398.

\bibitem{grace}
\textsc{J.\,H. Grace},
The zeros of a polynomial,
{\em Proc. Cambridge Philos. Soc.} \textbf{11} (1902), 352--357.

\bibitem{hormander}
\textsc{L.~H\"ormander},
On a theorem of Grace,
{\em Math. Scand.} \textbf{2} (1954), 55--64.

\bibitem{ismail-pos79}
\textsc{M.\,E.\,H. Ismail} and \textsc{M.\,V. Tamhankar},
A combinatorial approach to some positivity problems,
{\em SIAM J. Math. Anal.} \textbf{10} (1979), 478--485.

\bibitem{kauers-pos07}
\textsc{M.~Kauers},
Computer algebra and power series with positive coefficients,
in: {\em Formal Power Series and Algebraic Combinatorics} 2007
(Nankai University, Tianjin, China, 2007), 7~pp.;
available from
\texttt{http://\allowbreak igm.univ-mlv.fr/\allowbreak\~{}fpsac/\allowbreak FPSAC07/\allowbreak SITE07/\allowbreak PDF-Proceedings/\allowbreak Talks/\allowbreak20.pdf}.

\bibitem{kz-pos08}
\textsc{M.~Kauers} and \textsc{D.~Zeilberger},
Experiments with a positivity-preserving operator,
{\em Experiment. Math.} \textbf{17} (2008), 341--345.

\bibitem{koornwinder-pos78}
\textsc{T.~Koornwinder},
Positivity proofs for linearization and connection coefficients of orthogonal polynomials satisfying an addition formula,
{\em J. London Math. Soc.} (2) \textbf{18} (1978), 101--114.

\bibitem{maier-hyp07}
\textsc{R.\,S. Maier},
Algebraic hypergeometric transformations of modular origin,
{\em Trans. Amer. Math. Soc.} \textbf{359} (2007), 3859--3885.

\bibitem{pw-masy1}
\textsc{R.~Pemantle} and \textsc{M.\,C. Wilson},
Asymptotics of multivariate sequences. I. Smooth points of the singular variety,
{\rm J. Combin. Theory Ser. A} \textbf{97} (2002), 129--161.

\bibitem{pw-masy2}
\textsc{R.~Pemantle} and \textsc{M.\,C. Wilson},
Asymptotics of multivariate sequences. II. Multiple points of the singular variety,
{\em Combin. Probab. Comput.} \textbf{13} (2004), 735--761.

\bibitem{pw-masy-eg}
\textsc{R.~Pemantle} and \textsc{M.\,C. Wilson},
Twenty combinatorial examples of asymptotics derived from multivariate generating functions,
{\em SIAM Rev.} \textbf{50} (2008), 199--272.

\bibitem{rw-masy}
\textsc{A.~Raichev} and \textsc{M.\,C. Wilson},
Asymptotics of coefficients of multivariate generating functions: improvements for smooth points,
{\em Electron. J. Combin.} \textbf{15} (2008), R89, 17~pp.

\bibitem{ss-pos13}
\textsc{A.\,D. Scott} and \textsc{A.\,D. Sokal},
Complete monotonicity for inverse powers of some combinatorially defined polynomials,
{\em Preprint} \texttt{arXiv:\,1301.2449 [math.CO]} (2013), 79~pp.

\bibitem{straub-pos08}
\textsc{A.~Straub},
Positivity of {S}zeg\"o's rational function,
{\em Adv. Appl. Math.} \textbf{41} (2008), 255--264.

\bibitem{szego-pos33}
\textsc{G.~Szeg\"o},
{\"U}ber gewisse {P}otenzreihen mit lauter positiven {K}oeffizienten,
{\em Math. Z.} \textbf{37} (1933), 674--688.

\end{thebibliography}
\end{document}